\documentclass[10pt]{article}
\usepackage[margin=1in]{geometry}
\usepackage{amsthm, amsmath,amsfonts,amssymb,euscript,hyperref,graphics,color,slashed,mathrsfs}
\usepackage{graphicx}
\usepackage{comment}
\usepackage{import}
\usepackage{tikz}
\usepackage{latexsym}
\usepackage{mathtools}
\usepackage{appendix}

\newtheorem{theorem}{Theorem}[section]
\newtheorem{lemma}[theorem]{Lemma}

\newtheorem{definition}[theorem]{Definition}
\newtheorem{remark}[theorem]{Remark}

\numberwithin{equation}{section}

\begin{document}
\title {On the irrationality of certain $2$-adic zeta values}

\author{Li Lai}
\date{}

\maketitle

\begin{abstract}
	Let $\zeta_2(\cdot)$ be the Kubota–Leopoldt $2$-adic zeta function. We prove that, for every nonnegative integer $s$, there exists an odd integer $j$ in the interval $[s+3,3s+5]$ such that $\zeta_2(j)$ is irrational. In particular, at least one of $\zeta_2(7),\zeta_2(9),\zeta_2(11),\zeta_2(13)$ is irrational.
	
	Our approach is inspired by the recent work of Sprang. We construct explicit rational functions. The Volkenborn integrals of these rational functions' (higher-order) derivatives produce good linear combinations of $1$ and $2$-adic Hurwitz zeta values. The most difficult step is proving that certain Volkenborn integrals are nonzero, which is resolved by carefully manipulating the binomial coefficients. 
\end{abstract}


\section{Introduction}

It is well known that the values of the Riemann zeta function $\zeta(\cdot)$ at positive even integers are nonzero rational multiples of powers of $\pi$, and they are all transcendental. The arithmetic of values of $\zeta(\cdot)$ at positive odd integers ($\geqslant 3$) is more mysterious. Ap\'ery \cite{Ape79} achieved the first result toward this question by proving that $\zeta(3)$ is irrational. Rivoal \cite{Riv00}, Ball and Rivoal \cite{BR01} obtained a lower bound for the dimension of the $\mathbb{Q}$-space spanned by $1,\zeta(3),\zeta(5),\ldots,\zeta(s)$. As a corollary of the Ball-Rivoal theorem, there are infinitely many $\zeta(2k+1)$'s that are linearly independent over rational numbers. In recent years, some further developments have been made in \cite{FSZ19, LY20, Fis21}. Especially, Fischler \cite{Fis21} significantly improved the lower bound on the dimension. 

\bigskip

\newtheorem*{theoA}{Theorem A}
\begin{theoA}[Fischler \cite{Fis21}]
	For any sufficiently large odd integer $s$ we have
	\[\dim_{\mathbb{Q}} \operatorname{Span}_{\mathbb{Q}}\left(1,\zeta(3),\zeta(5),\ldots,\zeta(s)\right) \geqslant 0.21 \sqrt{\frac{s}{\log s}}.\]
\end{theoA}

\bigskip

Although $\zeta(3)$ is still the only particular odd zeta value known to be irrational, some partial irrationality results on $\zeta(s)$ with $s$ odd and small are known. Zudilin \cite{Zud01} proved that at least one of the four numbers $\zeta(5),\zeta(7),\zeta(9),\zeta(11)$ is irrational. Moreover, Zudilin \cite{Zud02} also proved that for every odd integer $s \geqslant 1$, the set $\{\zeta(j) \mid j \in \mathbb{Z} \cap [s+2,8s-1], j \text{~odd}\}$ contains at least one irrational number.

Let us turn our attention to the $p$-adic world. Throughout this paper, the letters $p$ and $q$ always denote prime numbers. The notations $\mathbb{Z}_p$, $\mathbb{Q}_p$ and $\mathbb{C}_p$ have their standard meanings. We use $v_p(x)$ to denote the $p$-adic order of $x$. The $p$-adic norm is defined by $|x|_p = p^{-v_p(x)}$. The notation $|y|$ denotes the Archimedean norm of $y$. The notation $\log$ always denotes the natural logarithm in the Archimedean world.

We are interested in the arithmetic of special values of the $p$-adic zeta function. The notation $\zeta_p(\cdot)$ has several slightly different meanings in the literature. We take the following one: for an integer $s \geqslant 2$, we define $\zeta_p(s) := L_p(s,\omega^{1-s})$, where $\omega$ is the Techm\"uller character and $L_p(s,\chi)$ is the Kubota–Leopoldt $p$-adic $L$-function associated to the character $\chi$. In this notation we have (for an integer $s \geqslant 2$):
\[\zeta_p(s) = \lim_{k \rightarrow s ~p\text{-adically}\atop k \in \mathbb{Z}_{<0},~~ k\equiv s \pmod{p-1}} \zeta(k) \quad \in \mathbb{Q}_p,\]
and $\zeta_p(s)$ vanishes when $s$ is a positive even integer. We will recall the definitions of $p$-adic Hurwitz zeta functions and $p$-adic $L$-functions in Section \ref{Sect_Prel}, through Volkenborn integrals.

Calegari \cite{Cal05} proved that $\zeta_p(3)$ is irrational for $p=2,3$. Very recently, Calegari, Dimitrov and Tang \cite{CDT20} proved that $\zeta_p(5)$ is irrational for $p=2$. They made use of the arithmetic holonomicity criterion. On the other hand, Sprang \cite{Spr20} proved the $p$-adic analogue of the Ball-Rivoal theorem.

\bigskip

\newtheorem*{theoB}{Theorem B}
\begin{theoB}[Sprang \cite{Spr20}]\label{thmB}
	Let $K$ be a number field (in $\mathbb{C}_p$). For $\varepsilon > 0$ and a sufficiently large positive odd integer $s$, we have
	\[\dim_{K} \operatorname{Span}_{K}\left(1,\zeta_p(3),\zeta_p(5),\ldots,\zeta_p(s)\right) \geqslant \frac{1-\varepsilon}{2[K:\mathbb{Q}](1+\log 2)} \log s.\]
\end{theoB}

\bigskip

Sprang \cite{Spr20} also obtained similar linear independence results for $p$-adic $L$-values and $p$-adic Hurwitz zeta values. For some other results on the irrationality of $p$-adic Hurwitz zeta values, see Beukers \cite{Beu08} and Bel \cite{Bel10,Bel19}.

Inspired by Sprang's work, we aim to establish the $p$-adic analogue of Zudilin's theorems mentioned above. We only address the case $p=2$ in the current paper. Our main result is as follows.

\medskip

\begin{theorem}\label{thm1}
	For every nonnegative integer $s$, the following set contains at least one irrational number:
	\[\left\{ \zeta_2(j) \mid j \in \mathbb{Z} \cap [s+3,3s+5], j~\text{odd} \right\}.\]
	In particular, at least one of the four $2$-adic numbers $\zeta_2(7),\zeta_2(9),\zeta_2(11),\zeta_2(13)$ is irrational.
\end{theorem}

\medskip

We also obtain similar irrationality results on $p$-adic zeta values for any odd prime $p$. However, the odd prime case is more involved. It will be presented in a subsequent joint paper. 

Our proof is actually working on the values of the $2$-adic Hurwitz zeta function $\zeta_2(\cdot,1/4)$. Note that for every odd integer $j \geqslant 3$ we have 
\[\zeta_2(j) = \frac{1}{4}\left(\zeta_2\left(j,\frac{1}{4}\right) + \zeta_2\left(j,\frac{3}{4}\right)\right) = \frac{1}{2}\zeta_2\left(j,\frac{1}{4}\right) \]
by the reflection formula $\zeta_p(j,x)=\zeta_p(j,1-x)$ of the $p$-adic Hurwitz zeta function. Although $\zeta_2(\cdot)$ vanishes at positive even integers, the values of $\zeta_2\left(\cdot,1/4\right)$ at positive even integers are nonzero and also interesting. We relax the parity restriction to prove the following two results.

\begin{theorem}\label{thm2}
	For any nonnegative integer $s$ and any $\delta \in \{0,1\}$, the following set contains at least one irrational number:
	\[ \left\{ \zeta_2\left(j,\frac{1}{4}\right) ~\mid~ j \in \mathbb{Z}\cap[s+3,3s+5],~j \equiv \delta \pmod{2} \right\}. \]
\end{theorem}

\begin{theorem}\label{thm3}
	For any nonnegative integer $s$, the following set contains at least one irrational number:
	\[ \left\{ \zeta_2\left(j,\frac{1}{4}\right) ~\mid~ j \in \mathbb{Z} \cap [s+3,2s+3] \right\}. \]
\end{theorem}

The special case $s=0$ of Theorem \ref{thm3} gives an alternative proof of the irrationality of $\zeta_2(3)$.


\bigskip

\noindent\textbf{Acknowledgements.} I wish to thank Johannes Sprang for several inspiring discussions in 2021 and 2023.

\bigskip

\subsection{An outline of the proof of Theorem \ref{thm2}}

Our approach is to construct a sequence of \emph{small but nonzero} linear forms in $1$ and $\zeta_2(j,1/4)$ ($s+3 \leqslant j \leqslant 3s+5$, $j \equiv s+1+\delta \pmod{2}$). Then Lemma \ref{Lem_2_1} would imply the desired irrationality result.

We first construct a sequence of rational functions:
\[ A_n(t) := 2^{(6s+12)n} \cdot (4t+2n)^{\delta} \cdot \frac{\left(t+\frac{1}{4}\right)_{n}^{s+2}\left(t+\frac{3}{4}\right)_{n}^{s+2}}{(t)_{n+1}^{2s+4}}. \]
Then the Volkenborn integrals of the $s$-th order derivatives of $A_n(t+1/4)$ ($n=1,2,\ldots$) produce the linear forms we need:
\begin{align*}
S_n &:= \int_{\mathbb{Z}_2} A_n^{(s)}\left(t+\frac{1}{4}\right) \mathrm{d}t  \\
&= \rho_{n,0} + \sum_{2 \leqslant i \leqslant 2s+4 \atop i \equiv \delta \pmod{2}} \rho_{n,i} ~ \zeta_2\left( i+s+1, \frac{1}{4} \right).
\end{align*}

Using familiar techniques in the Archimedean case, we can show that for every large $n$:
\[ \Phi_n^{-s-2}d_n^{3s+5} \rho_{n,i} \in \mathbb{Z}, \quad 0 \leqslant i \leqslant 2s+4 \]
and 
\[ \max_{0\leqslant i \leqslant 2s+4} |\Phi_n^{-s-2}d_n^{3s+5}\rho_{n,i}| \leqslant \exp\left( \left( (4+4\log 2)s +7+ 8\log 2 +o(1) \right)n \right). \]
Here $d_n = \operatorname{lcm}[1,2,\ldots,n]$ and $\Phi_n$ is a product of certain primes (see \eqref{def_Phi}).

For the $p$-adic aspect, it is not hard to bound the $2$-adic norm of the linear forms from above for every large $n$:
\[ |\Phi_n^{-s-2}d_n^{3s+5}S_n|_2 \leqslant 2^{(-10s-20+o(1))n}. \]
So the linear forms are small enough. But a serious difficulty is to prove that $S_n \neq 0$. It is reasonable to conjecture that $S_n \neq 0$ for every $n$; however, it seems quite hard to prove so. 

In order to apply Lemma \ref{Lem_2_1}, we only need $S_n \neq 0$ for infinitely many $n$. We adapt Sprang's method in \cite{Spr20} to our situation and prove that $S_n \neq 0$ when $n=2^m-1$, $m \geqslant 2$. It eventually relies on the arithmetic properties of the binomial coefficients. More precisely, after applying the Leibniz rule we can write 
\[ A_n^{(s)}\left(t+\frac{1}{4}\right)  = \sum_{\underline{i}\in I} f_{\underline{i}}(t) \]
as a sum of finitely many functions. We find a particular index $\underline{i}^{*}$ such that $f_{\underline{i}^{*}}(t)$ dominates the sum  in the sense that:

\begin{align}
&\bullet\quad \int_{\mathbb{Z}_2}f_{\underline{i}^{*}}(t) \mathrm{d}t \neq 0, \label{nonv}\\
&\bullet\quad v_2\left( \int_{\mathbb{Z}_2}f_{\underline{i}}(t) \mathrm{d}t  \right) > v_2\left( \int_{\mathbb{Z}_2}f_{\underline{i}^{*}}(t) \mathrm{d}t  \right) \quad\text{~for every~} \underline{i} \neq \underline{i}^{*}. \label{domi}
\end{align}
The first key point \eqref{nonv} is established by the adaptation of Sprang's method, see Lemma \ref{Lem_2_3} and Lemma \ref{Lem_2_4}. On the other hand, an elementary observation, Lemma \ref{Lem_5_1},  plays the central role for \eqref{domi}.

\bigskip

\section{Preliminaries}\label{Sect_Prel}

\subsection{An elementary irrationality criterion}
A basic tool for proving irrationality or linear independence is to construct small but nonzero linear forms. The following elementary lemma is sufficient for our needs.

\begin{lemma}\label{Lem_2_1}
	Let $\alpha_1,\alpha_2,\ldots,\alpha_m \in \mathbb{Q}_p$ and let $a_{n,0},a_{n,1},\ldots,a_{n,m} \in \mathbb{Z}$ ($n=1,2,3,\ldots$) be $m+1$ sequences of integers such that
	\begin{equation}\label{eqn_2_1}
	\lim_{n \rightarrow \infty} \left( \max_{0\leqslant j \leqslant m} |a_{n,j}| \cdot \left|a_{n,0}+\sum_{j=1}^{m} a_{n,j}\alpha_j \right|_p \right) = 0 
	\end{equation}
	and 
	\[a_{n,0}+\sum_{j=1}^{m} a_{n,j}\alpha_j  \neq 0 \quad\text{infinitely often}.\]
	Then at least one of $\alpha_1,\alpha_2,\ldots,\alpha_m$ is irrational.
\end{lemma}

\begin{proof}
	Passing to a subsequence, we may assume that $a_{n,0}+\sum_{j=1}^{m} a_{n,j}\alpha_j  \neq 0$ for every $n$. (In particular, this implies $\max_{0\leqslant j \leqslant m} |a_{n,j}| \neq 0$.) Suppose that all $\alpha_j$ are rational numbers and let $D>0$ be a common denominator of $\alpha_j$. Then $D\cdot\left(a_{n,0}+\sum_{j=1}^{m} a_{n,j}\alpha_j\right)$ is a nonzero integer, so we have trivially
	\begin{align*} \left| D\cdot\left(a_{n,0}+\sum_{j=1}^{m} a_{n,j}\alpha_j\right) \right|_p &\geqslant \frac{1}{\left| D\cdot\left(a_{n,0}+\sum_{j=1}^{m} a_{n,j}\alpha_j\right) \right|} \\ &\geqslant  \frac{1}{D\cdot\left(1+ \sum_{j=1}^{m}|\alpha_j|\right)\cdot \max_{0\leqslant j \leqslant m} |a_{n,j}|},
	\end{align*}
	which implies 
	\[ \max_{0\leqslant j \leqslant m} |a_{n,j}| \cdot \left|a_{n,0}+\sum_{j=1}^{m} a_{n,j}\alpha_j \right|_p \geqslant \frac{1}{|D|_p \cdot D\cdot\left(1+ \sum_{j=1}^{m}|\alpha_j|\right)}.\]
	Since the right-hand side above is independent of $n$, we obtain a contradiction to \eqref{eqn_2_1}. Thus we conclude that at least one of $\alpha_1,\alpha_2,\ldots,\alpha_m$ is irrational. 
\end{proof}

\subsection{Volkenborn integrals}

In this subsection, we briefly recall the definitions of strictly differentiable functions and Volkenborn integrals. We refer the reader to Robert's textbook \cite[Chapter 5]{Rob00} for more details. 

Let $K$ be either $\mathbb{Z}_p$ or $\mathbb{Q}_p$. A function $f:\mathbb{Z}_p \rightarrow K$ is said to be \emph{strictly differentiable} on $\mathbb{Z}_p$ -- denoted by $f \in S^{1}(\mathbb{Z}_p,K)$ -- if 
\[f(x) - f(y) = (x-y) g(x,y)\]
for some continuous function $g(x,y)$ on $\mathbb{Z}_p \times \mathbb{Z}_p$. As a basic example, every convergent power series on $\mathbb{Z}_p$ is strictly differentiable on $\mathbb{Z}_p$.

A function $f: \mathbb{Z}_p \rightarrow \mathbb{Q}_p$ is said to be \emph{Volkenborn integrable} \cite{Vol72} if the sequence
\[\frac{1}{p^n} \sum_{k=0}^{p^n-1} f(k)\]
converges $p$-adically as $n \rightarrow \infty$. In this case, the value
\[ \int_{\mathbb{Z}_p} f(t)\mathrm{d}t := \lim_{n \rightarrow \infty} \frac{1}{p^n} \sum_{k=0}^{p^n-1} f(k) \]
is called the \emph{Volkenborn integral} of $f$. As a basic example, every strictly differentiable function on $\mathbb{Z}_p$ is Volkenborn integrable. 

The Volkenborn integral is not translation-invariant; indeed, we have the following.

\begin{lemma}\label{Lem_translation_formula}
	Let $f \in S^{1}\left(\mathbb{Z}_p,\mathbb{Q}_p\right)$ and let $k$ be a positive integer. We have
	\[ \int_{\mathbb{Z}_p} f(t+k)\mathrm{d}t = \int_{\mathbb{Z}_p} f(t)\mathrm{d}t + \sum_{\ell=0}^{k-1} f'(\ell). \]
\end{lemma}
\begin{proof}
	It follows directly from the definition. See also \cite[pp. 267--268]{Rob00}.
\end{proof}

For any positive integer $k$, we denote by $k_{-}$ the nonnegative integer obtained by deleting the leading digit in the $p$-adic expansion of $k$. In other words, if $k=a_0+a_1p+\cdots+a_{l}p^{l}$, $a_0,\ldots,a_l \in \{0,1,\ldots,p-1\}$, $a_l \neq 0$, then $k_{-} = a_0+a_1p+\cdots+a_{l-1}p^{l-1}$. For the purpose of estimating the $p$-adic norm of Volkenborn integrals, we introduce the following notations, which are slightly different from the notations used by Sprang in \cite{Spr20}. 

\begin{definition}
Let $f: \mathbb{Z}_p \rightarrow \mathbb{Q}_p$ be a function. For every nonnegative integer $m$, we define
\[\triangle_m(f) := \inf_{k \geqslant p^m} v_p\left( \frac{f(k)-f(k_{-})}{k-k_{-}} \right). \]
We also define
\[\triangle(f) := \inf\{ 1+v_p(f(0)),\triangle_{0}(f) \}.\]
\end{definition}
From the definition, it is clear that 
\[-\infty \leqslant \triangle(f) \leqslant \triangle_0(f) \leqslant \triangle_1(f) \leqslant \triangle_2(f) \leqslant  \triangle_3(f) \leqslant \cdots \leqslant +\infty,\]
and $\triangle(f) > -\infty$ if $f$ is strictly differentiable on $\mathbb{Z}_p$. For a constant function $f$, we have $\triangle_0(f) = +\infty$.

\begin{lemma}\label{Lem_2_3}
	Let $f \in S^{1}(\mathbb{Z}_p,\mathbb{Q}_p)$. Then
	\begin{equation}\label{eqn_2_2}
	 v_p\left( \int_{\mathbb{Z}_p} f(t)\mathrm{d}t \right) \geqslant \triangle(f) - 1.
	 \end{equation}
	Moreover, for every nonnegative integer $m$ we have
	\begin{equation}\label{eqn_2_3}
	\int_{\mathbb{Z}_p} f(t)\mathrm{d}t \equiv \frac{1}{p^m}\sum_{k=0}^{p^m-1} f(k) \pmod{p^{\triangle_m(f)-1}\mathbb{Z}_p}.
	\end{equation}
\end{lemma}

\begin{proof}
	It is straightforward to check that, for any integer $M > m$ we have
	\[\frac{1}{p^M} \sum_{k=0}^{p^M-1} f(k) = \frac{1}{p^m}\sum_{k=0}^{p^m-1} f(k) + \sum_{l=m}^{M-1}\sum_{k=p^l}^{p^{l+1}-1} \frac{f(k)-f(k_{-})}{p^{l+1}}. \] 
	For any integer $k \in[p^l,p^{l+1}-1]$ ($l \geqslant m$) we have 
	\[v_p\left( \frac{f(k)-f(k_{-})}{p^{l+1}} \right) \geqslant \triangle_m(f) -1\]
	by the definition of $\triangle_m(f)$. Therefore, 
	\[ \frac{1}{p^M} \sum_{k=0}^{p^M-1} f(k) \equiv \frac{1}{p^m}\sum_{k=0}^{p^m-1} f(k) \pmod{p^{\triangle_m(f)-1}\mathbb{Z}_p}. \]
	Taking the limit as $M\rightarrow \infty$ we obtain \eqref{eqn_2_3}.
	
	In particular, 
	\[\int_{\mathbb{Z}_p} f(t)\mathrm{d}t \equiv f(0) \pmod{p^{\triangle_0(f)-1}\mathbb{Z}_p}. \]
	Now \eqref{eqn_2_2} follows immediately from the definition of $\triangle(f)$.
\end{proof}

\begin{lemma}\label{Lem_2_4}
	We have the following properties for the operators $\triangle$ and $\triangle_m$ ($m=0,1,2,\ldots$).
	\begin{enumerate}
		\item[(1)] If $f(t) = \sum_{j=0}^{\infty} a_jt^{j} \in \mathbb{Z}_{p}[[t]]$ and $\lim_{j \rightarrow \infty} |a_j|_p = 0$, then $\triangle_m(f) \geqslant \triangle(f) \geqslant 0$.
		\item[(2)] If $f,g \in S^{1}(\mathbb{Z}_{p},\mathbb{Z}_{p})$, then $\triangle(f\cdot g) \geqslant \min\{\triangle(f),\triangle(g)\}$ and $\triangle_m(f\cdot g) \geqslant \min\{\triangle_m(f),\triangle_m(g)\}$.
		\item[(3)] For $n,j \in \mathbb{Z}$, $n>0$, and $f(t) = \binom{t+j}{n}$, we have $\triangle_m(f) \geqslant \triangle(f) \geqslant -\left\lfloor \frac{\log n}{\log p} \right\rfloor$. Moreover, for $m > \left\lfloor \frac{\log n}{\log p} \right\rfloor$ we have
		\[\triangle_m(f^p) \geqslant -\left\lfloor \frac{\log n}{\log p} \right\rfloor + 1.\]
	\end{enumerate} 
\end{lemma}

\begin{proof}

$(1)$ Clearly $v_p(f(0)) \geqslant 0$. For any positive integer $k$ we notice that
		\[\frac{f(k)-f(k_{-})}{k-k_{-}} = \sum_{j=1}^{\infty} a_j \sum_{i=0}^{j-1} k^{i}k_{-}^{j-1-i}\]
is absolutely convergent to an element in $\mathbb{Z}_p$. So $\triangle_m(f) \geqslant \triangle(f) \geqslant 0$ for any nonnegative integer $m$.
		
$(2)$ Since $f(0),g(0) \in \mathbb{Z}_p$, we have $v_p(f(0)g(0)) \geqslant \min\{v_p(f(0)),v_p(g(0))\}$. For any positive integer $k$, we have
		\[\frac{f(k)g(k)-f(k_{-})g(k_{-})}{k-k_{-}} = g(k_{-})\cdot\frac{f(k)-f(k_{-})}{k-k_{-}} + f(k)\cdot\frac{g(k)-g(k_{-})}{k-k_{-}}.\]
Since $g(k_{-}), f(k) \in \mathbb{Z}_p$, we have 
		\[v_p\left(\frac{f(k)g(k)-f(k_{-})g(k_{-})}{k-k_{-}}\right) \geqslant \min\left\{ v_p\left(\frac{f(k)-f(k_{-})}{k-k_{-}}\right), v_p\left(\frac{g(k)-g(k_{-})}{k-k_{-}}\right) \right\}.\]
Now the conclusion follows easily.
		
$(3)$ Clearly $v_p(f(0)) \geqslant 0$. For any positive integer $k$, using the well-known identity
		\[\binom{x+y}{n} = \sum_{i=0}^{n} \binom{x}{i}\binom{y}{n-i}\]
we obtain
		\[ \frac{f(k)-f(k_{-})}{k-k_{-}} = \frac{1}{k-k_{-}} \sum_{i=1}^{n} \binom{k-k_{-}}{i}\binom{k_{-}+j}{n-i} = \sum_{i=1}^{n} \frac{1}{i}\binom{k-k_{-}-1}{i-1}\binom{k_{-}+j}{n-i}. \]
Therefore, 
		\[v_p\left( \frac{f(k)-f(k_{-})}{k-k_{-}} \right) \geqslant \min_{1 \leqslant i \leqslant n} v_p\left( \frac{1}{i} \right) = -\left\lfloor \frac{\log n}{\log p} \right\rfloor.\]
We conclude that for any nonnegative integer $m$ we have $\triangle_m(f) \geqslant \triangle(f) \geqslant -\left\lfloor \frac{\log n}{\log p} \right\rfloor$.
		
Suppose now $m > \left\lfloor \frac{\log n}{\log p} \right\rfloor$. For any integer $k \geqslant p^m$, we have $v_p(k-k_{-}) \geqslant m > \left\lfloor \frac{\log n}{\log p} \right\rfloor$ and $v_p(f(k)-f(k_{-})) > 0$. So $f(k) \equiv f(k_{-}) \pmod{p}$ and 
\[\sum_{i=0}^{p-1} f(k)^{i}f(k_{-})^{p-1-i} \equiv p \cdot f(k)^{p-1} \equiv 0 \pmod{p}. \]
Therefore, for any integer $k \geqslant p^m$ we have
\begin{align*}
v_p\left( \frac{f(k)^p - f(k_{-})^p}{k - k_{-}} \right) &= v_p\left( \frac{f(k)-f(k_{-})}{k-k_{-}} \cdot \sum_{i=0}^{p-1} f(k)^{i}f(k_{-})^{p-1-i} \right) \\
&\geqslant -\left\lfloor \frac{\log n}{\log p} \right\rfloor + 1,
\end{align*}
which implies 
\[\triangle_m(f^p) \geqslant -\left\lfloor \frac{\log n}{\log p} \right\rfloor + 1\]
for any integer $m>\left\lfloor \frac{\log n}{\log p} \right\rfloor$.

\end{proof}

\subsection{The $p$-adic Hurwitz zeta functions and $L$-functions}
In this subsection, we briefly recall the definitions of $p$-adic Hurwitz zeta functions and $p$-adic $L$-functions. We refer the reader to Cohen's textbook \cite[Chapter 11]{Coh07} for more details. 

We set $q_p = p$ if $p$ is an odd prime, and $q_2 = 4$. The units $\mathbb{Z}_p^{\times}$ of the $p$-adic integers decompose canonically as
\[ \mathbb{Z}_p^{\times} \cong \mu_{\varphi(q_p)}\left(\mathbb{Z}_p\right) \times \left(1+q_p\mathbb{Z}_p\right).  \]
Here $\mu_n(R)$ denotes the group of $n$-th roots of unity in a ring $R$ and $\varphi(\cdot)$ denotes Euler's totient function. The canonical projection
\[\omega: \mathbb{Z}_p^{\times} \rightarrow \mu_{\varphi(q_p)}\left(\mathbb{Z}_p\right)\]
is called the \emph{Techm\"uller character}. We extend $\omega$ to a map $\mathbb{Q}_p^{\times} \rightarrow \mathbb{Q}_p^{\times}$ by setting
\[\omega(x) := p^{v_p(x)} \omega\left( \frac{x}{p^{v_p(x)}} \right) \]
and define $\langle x \rangle := x/\omega(x)$. 

For $s \in \mathbb{C}_p \setminus \{1\}$ such that $|s|_p < q_pp^{-1/(p-1)}$ and $x \in \mathbb{Q}_p$ such that $|x|_p \geqslant q_p$, we define 
\[ \zeta_p(s,x) := \frac{1}{s-1}\int_{\mathbb{Z}_p} \langle t+x \rangle^{1-s} \mathrm{d}t. \]
For fixed $x \in \mathbb{Q}_p$ such that $|x|_p \geqslant q_p$, the \emph{$p$-adic Hurwitz zeta function} $\zeta_p(\cdot,x)$ is a $p$-adic meromorphic function on $|s|_p < q_pp^{-1/(p-1)}$, which in addition is analytic, except for a simple pole at $s = 1$ with residue $1$. Moreover, we have the reflection formula \cite[Theorem 11.2.9]{Coh07}:
\begin{equation}\label{eqn_reflection_formula}
\zeta_p(s,x) =\zeta_p(s,1-x).
\end{equation}

\medskip

\begin{lemma}\label{Lem_2_6}
	Let $x \in \mathbb{Q}_p$ such that $|x|_p \geqslant q_p $. For any integer $j \geqslant 1$ we have
	\[  \int_{\mathbb{Z}_p} \frac{\mathrm{d}t}{(t+x)^j} = j\cdot\omega(x)^{-j}\cdot\zeta_p(j+1,x). \]
	In particular, 
	\[\int_{\mathbb{Z}_2} \frac{\mathrm{d}t}{\left(t+\frac{1}{4}\right)^j} = j4^j\cdot\zeta_2\left(j+1,\frac{1}{4}\right). \]
\end{lemma}

\begin{proof}
	It follows directly from the definition. See also \cite[Propsition 11.2.6]{Coh07}.
\end{proof}

\medskip

Let $\chi$ be a primitive character of conductor $f$. Let $M$ be a common multiple of $f$ and $q_p$. For $s \in \mathbb{C}_p\setminus\{1\}$ such that $|s|_p<q_pp^{-1/(p-1)}$, we define 
\[ L_p(s,\chi) := \frac{\langle M \rangle^{1-s}}{M}\sum_{a=0\atop p \nmid a}^{M-1} \chi(a)\zeta_p\left(s,\frac{a}{M}\right). \]
It can be shown that the above definition is independent of the choice of $M$. The \emph{$p$-adic $L$-function} $L_p(\cdot,\chi)$ is a $p$-adic analytic function on $|s|_p < q_pp^{-1/(p-1)}$, except when $\chi=\chi_0$ is the trivial character, in which case $L_p(\cdot,\chi_0)$ has a simple pole at $s=1$ with residue $1-1/p$. (See \cite[Propsition 11.3.9]{Coh07}.)

The functions $L_p(\cdot,\omega^{0}), L_p(\cdot,\omega^{1}),\ldots, L_p(\cdot,\omega^{p-2})$ are exactly the $p-1$ branches interpolating the values of
the Riemann zeta function at negative integers. In the current paper, we define $\zeta_p(s) := L_p(s,\omega^{1-s})$ for an integer $s \geqslant 2$.
\begin{lemma}\label{Lem_2_7}
	For any odd integer $j \geqslant 3$, we have
	\[ \zeta_p(j) = \frac{1}{q_p} \sum_{a=0 \atop p \nmid a}^{q_p-1} \omega(a)^{1-j} \zeta_p\left( j,\frac{a}{q_p} \right) = \frac{2}{q_p} \sum_{0<a<q_p/2} \omega(a)^{1-j} \zeta_p\left( j,\frac{a}{q_p} \right). \]
	In particular, for any odd integer $j \geqslant 3$ we have
	\[ \zeta_2(j) = \frac{1}{2}\zeta_2\left(j,\frac{1}{4}\right). \]
\end{lemma}
\begin{proof}
	It follows directly from the definition and the reflection formula \eqref{eqn_reflection_formula}.
\end{proof}

\section{Rational functions and linear forms}

For a positive integer $k$, we denote the Pochhammer symbol by
\[(\alpha)_k:=\alpha(\alpha+1)(\alpha+2)\cdots(\alpha+k-1).\]
For $k=0$, we define $(\alpha)_0 := 1$.

\begin{definition}
	Fix a nonnegative integer $s$ and fix a choice of $\delta \in \{0,1\}$. For every positive integer $n$ we define the following two rational functions.
	\begin{align}
		A_n(t) &:= 2^{(6s+12)n} \cdot (4t+2n)^{\delta} \cdot \frac{\left(t+\frac{1}{4}\right)_{n}^{s+2}\left(t+\frac{3}{4}\right)_{n}^{s+2}}{(t)_{n+1}^{2s+4}}, \label{def_A}\\
		B_n(t) &:= 2^{(3s+6)n} \cdot\frac{\left(t+\frac{3}{4}\right)_{n}^{s+2}}{(t)_{n+1}^{s+2}}. \label{def_B}
	\end{align}
For a rational function $R(t) = P(t)/Q(t)$, we define its degree by $\deg R := \deg P - \deg Q$. Both $A_n(t)$ and $B_n(t)$ have degree $\leqslant -2$. We define their partial fraction decompositions by
\begin{align}
	A_n(t) &=: \sum_{i=1}^{2s+4}\sum_{k=0}^{n} \frac{a_{n,i,k}}{(t+k)^i}, \label{def_a}\\
	B_n(t) &=: \sum_{i=1}^{s+2}\sum_{k=0}^{n}\frac{b_{n,i,k}}{(t+k)^i}. \label{def_b}
\end{align}
\end{definition}
\medskip

For a function $f(t)$, we denote by $f^{(s)}(t)$ the $s$-th order derivative of $f(t)$. The Volkenborn integrals of $A_n^{(s)}(t+1/4)$ (respectively, $B_n^{(s)}(t+1/4)$) ($n=1,2,\ldots$) produce good linear forms to prove Theorem \ref{thm2} (respectively, Theorem \ref{thm3}).

\begin{definition}
	For every positive integer $n$ we define
	\begin{align}
		S_n := \int_{\mathbb{Z}_2} A_n^{(s)}\left(t+\frac{1}{4}\right) \mathrm{d}t ~\in \mathbb{Q}_2, \label{def_S}\\
		T_n := \int_{\mathbb{Z}_2} B_n^{(s)}\left(t+\frac{1}{4}\right) \mathrm{d}t ~\in \mathbb{Q}_2. \label{def_T}
	\end{align}
\end{definition}
\medskip

\begin{lemma}\label{Lin}
	For every positive integer $n$ we have
	\begin{align}
		S_n &= \rho_{n,0} + \sum_{2 \leqslant i \leqslant 2s+4 \atop i \equiv \delta \pmod{2}} \rho_{n,i} ~ \zeta_2\left( i+s+1, \frac{1}{4} \right), \label{eqn_S_rho}\\
		T_n &= \sigma_{n,0} + \sum_{i=2}^{s+2} \sigma_{n,i} ~ \zeta_2\left( i+s+1, \frac{1}{4} \right), \label{eqn_T_sigma}
	\end{align}
where
\begin{align}
	\rho_{n,0} &= (-1)^{s+1} \sum_{i=1}^{2s+4}\sum_{k=1}^{n}\sum_{\ell=0}^{k-1} \frac{(i)_{s+1}a_{n,i,k}}{\left(\ell+\frac{1}{4}\right)^{i+s+1}} ~\in\mathbb{Q}, \label{def_rho_0}\\
	\rho_{n,i} &= (-1)^{s}(i)_{s+1}4^{i+s} \sum_{k=0}^{n} a_{n,i,k} ~\in\mathbb{Q}, \quad (1 \leqslant i \leqslant 2s+4) \label{def_rho_i}\\
	\sigma_{n,0} &= (-1)^{s+1} \sum_{i=1}^{s+2}\sum_{k=1}^{n}\sum_{\ell=0}^{k-1} \frac{(i)_{s+1}b_{n,i,k}}{\left(\ell+\frac{1}{4}\right)^{i+s+1}} ~\in\mathbb{Q}, \label{def_sigma_0}\\
	\sigma_{n,i} &= (-1)^{s}(i)_{s+1}4^{i+s} \sum_{k=0}^{n} b_{n,i,k} ~\in\mathbb{Q}. \quad (1 \leqslant i \leqslant s+2) \label{def_sigma_i}
\end{align}
\end{lemma}

\begin{proof}
By \eqref{def_a}, we have 
\[ A_n^{(s)}\left(t+\frac{1}{4}\right) = (-1)^{s} \sum_{i=1}^{2s+4}\sum_{k=0}^{n} \frac{(i)_{s}a_{n,i,k}}{\left(t+k+\frac{1}{4}\right)^{i+s}}. \]	
Therefore, 
\begin{align*}
	S_n &= (-1)^{s} \sum_{i=1}^{2s+4}\sum_{k=0}^{n} (i)_{s}a_{n,i,k} \int_{\mathbb{Z}_2}  \frac{\mathrm{d}t}{\left(t+k+\frac{1}{4}\right)^{i+s}} \\
	&\overset{\text{Lemma \ref{Lem_translation_formula}}}{=} (-1)^{s} \sum_{i=1}^{2s+4}\sum_{k=0}^{n} (i)_{s}a_{n,i,k} \left( \int_{\mathbb{Z}_2}  \frac{\mathrm{d}t}{\left(t+\frac{1}{4}\right)^{i+s}} -(i+s)\sum_{\ell=0}^{k-1} \frac{1}{\left(\ell+\frac{1}{4}\right)^{i+s+1}} \right) \\
	&\overset{\text{Lemma \ref{Lem_2_6}}}{=} (-1)^{s} \sum_{i=1}^{2s+4}\sum_{k=0}^{n} (i)_{s}a_{n,i,k} \left( (i+s)4^{i+s} \zeta_2\left(i+s+1,\frac{1}{4}\right) -(i+s)\sum_{\ell=0}^{k-1} \frac{1}{\left(\ell+\frac{1}{4}\right)^{i+s+1}} \right) \\
	&= \rho_{n,0} + \sum_{i=1}^{2s+4} \rho_{n,i} ~ \zeta_2\left( i+s+1, \frac{1}{4} \right).
\end{align*}
Here when $k=0$, the sum $\sum_{\ell=0}^{k-1}$ is understood as $0$.
Since $\deg A_n(t) \leqslant -2$, we have 
\[ \rho_{n,1} = (-1)^{s}(s+1)!4^{s+1} \cdot \lim_{t \rightarrow \infty} tA_n(t) = 0 .\]
The rational function $A_n(t)$ has the symmetry
\[ A_n(-t-n) = (-1)^{\delta} A_n(t), \]
which implies $(-1)^{i}a_{n,i,n-k} = (-1)^{\delta}a_{n,i,k}$ for all $i,k$. So $\rho_{n,i} = 0$ when $i>0$ and $i \not\equiv \delta \pmod{2}$. In conclusion, we have
\[ S_n = \rho_{n,0} + \sum_{i=1}^{2s+4} \rho_{n,i} ~ \zeta_2\left( i+s+1, \frac{1}{4} \right) = \rho_{n,0} + \sum_{2 \leqslant i \leqslant 2s+4 \atop i \equiv \delta \pmod{2}} \rho_{n,i} ~ \zeta_2\left( i+s+1, \frac{1}{4} \right). \]
This completes the proof of \eqref{eqn_S_rho}. The proof of \eqref{eqn_T_sigma} is similar and simpler: we have $\sigma_{n,1}=0$ since $\deg B_n(t) \leqslant -2$, and we do not need to consider the parity.
\end{proof}

\section{Arithmetic properties of the coefficients}

As usual, we denote by 
\[ d_n := \operatorname{lcm}[1,2,\ldots,n] \]
the least common multiple of the smallest $n$ positive integers.

\begin{lemma}\label{Lem_4_1}
	Let 
	\begin{align*}
		F_{1/4}(t) &:= \frac{2^{3n}}{n!} \left( t+\frac{1}{4} \right)_n, \\
		F_{3/4}(t) &:= \frac{2^{3n}}{n!} \left( t+\frac{3}{4} \right)_n, \\
		G(t)  &:= \frac{n!}{(t)_{n+1}}.
	\end{align*}
Then for any integer $k \in [0,n]$ and any nonnegative integer $\ell$, we have
\begin{align}
	&d_n^{\ell} \cdot \frac{1}{\ell !}F_{1/4}^{(\ell)}(t) \big|_{t = -k} \in \mathbb{Z}, \label{eqn_4_1}\\
	&d_n^{\ell} \cdot \frac{1}{\ell !}F_{3/4}^{(\ell)}(t) \big|_{t = -k} \in \mathbb{Z}, \label{eqn_4_2}\\
	&d_n^{\ell} \cdot \frac{1}{\ell !} \left( (t+k)G(t) \right)^{(\ell)} \big|_{t =-k} \in \mathbb{Z}. \label{eqn_4_3}
\end{align}
\end{lemma}

\begin{proof}
	These are well-known. For details, we refer the reader to \cite[Propsition 3.2]{LY20} for \eqref{eqn_4_1} and \eqref{eqn_4_2}; we refer the reader to \cite[Lemma 16]{Zud04} for \eqref{eqn_4_3}.
\end{proof}

\begin{lemma}\label{Lem_4_2}
	We have 
	\begin{align}
		&d_n^{2s+4-i} a_{n,i,k} \in \mathbb{Z}, \quad (1 \leqslant i \leqslant 2s+4,~0 \leqslant k \leqslant n) \label{weak_ari_a}\\
		&d_n^{s+2-i} b_{n,i,k} \in \mathbb{Z}. \quad (1 \leqslant i \leqslant s+2,~0 \leqslant k \leqslant n) \label{ari_b}
	\end{align}
\end{lemma}

\begin{proof}
	We only prove \eqref{weak_ari_a}, the proof of \eqref{ari_b} is similar and simpler. By \eqref{def_a} we have
	\begin{equation}\label{eqn_a=deri}
	 a_{n,i,k} = \frac{1}{(2s+4-i)!} \left( (t+k)^{2s+4}A_n(t) \right)^{(2s+4-i)} \big|_{t=-k}. 
	 \end{equation}
	Keep the notations in Lemma \ref{Lem_4_1}. Let $H(t)=(4t+2n)^{\delta}$. Obviously, we have
	\[ \frac{1}{\ell !} H^{(\ell)}(t) \big|_{t=-k} \in \mathbb{Z} \]
for any integer $k\in[0,n]$ and any nonnegative integer $\ell$. 	
	
	Note that 
	\[ (t+k)^{2s+4}A_n(t) = H(t) \cdot F_{1/4}(t)^{s+2} \cdot F_{3/4}(t)^{s+2} \cdot \left((t+k)G(t)\right)^{2s+4}. \]
	Applying the Leibniz rule and Lemma \ref{Lem_4_1}, we have
	\begin{align*}
		&d_n^{2s+4-i} a_{n,i,k} \\
		=&  \sum ~~ \frac{d_n^{h}H^{(h)}(t)}{h!} \cdot \prod_{u=1}^{s+2} \frac{d_n^{f_{u}}F_{1/4}^{(f_{u})}(t)}{f_{u}!} \cdot \prod_{u=1}^{s+2} \frac{d_n^{f_{u}^{\prime}}F_{3/4}^{(f_{u}^{\prime})}(t)}{f_{u}^{\prime}!} \cdot \prod_{u=1}^{2s+4} \frac{d_n^{g_u}\left((t+k)G(t)\right)^{(g_u)}}{g_u !} \big|_{t=-k} \in \mathbb{Z},		
	\end{align*} 
where the sum is taken over all tuples
\[ (h,f_{1},\ldots,f_{s+2},f_{1}^{\prime},\ldots,f_{s+2}^{\prime},g_{1},\ldots,g_{2s+4}) \]
of nonnegative integers such that the sum of entries equals $2s+4-i$. 
\end{proof}

We refine \eqref{weak_ari_a} in the following lemma. This is necessary for the purpose of proving Theorem \ref{thm2}.

\begin{lemma}\label{Lem_4_3}
	For any positive integer $n$ we define 
	\begin{equation}\label{def_Phi}
	\Phi_n := \prod_{\sqrt{10n}<q\leqslant n \atop \left\{ \frac{n}{q} \right\} > \frac{1}{2}} q .
	\end{equation}
Here $\left\{  n/q \right\}$ is the fractional part of $n/q$. When $n > (2s+4)^2$ and $n$ is odd, we have
\begin{equation}\label{strong_ari_a}
	\Phi_n^{-s-2} d_n^{2s+4-i} a_{n,i,k} \in \mathbb{Z}. \quad (1 \leqslant i \leqslant 2s+4,~0 \leqslant k \leqslant n)
\end{equation}
\end{lemma}

\begin{proof}
Let $j = 2s+4 -i \in [0,2s+3]$. By \eqref{weak_ari_a}, our goal \eqref{strong_ari_a} is reduced to the following assertion: for any prime $q$ such that $\sqrt{10n}<q\leqslant n$ and $\left\{ n/q\right\} > 1/2$, for any $0 \leqslant j \leqslant 2s+3$ and any $0 \leqslant k \leqslant n$, we have
\begin{equation}\label{asser}
	v_q(a_{n,2s+4-j,k}) \geqslant s+2-j.
\end{equation}
  
We fix $q$ and $k$, and prove the assertion \eqref{asser} by induction on $j$. 

Let $\widetilde{A_n}(t):=(t+k)^{2s+4}A_n(t)$. Then \eqref{eqn_a=deri} can be rephrased as
\begin{equation}\label{eqn_a=deri_new}
 a_{n,2s+4-j,k} = \frac{1}{j!}\widetilde{A_n}^{(j)}(-k).
\end{equation} 
For $j=0$, we have
\begin{align*}
	a_{n,2s+4,k} &= \widetilde{A_n}(-k) \\ &= 2^{(6s+12)n}(-4k+2n)^{\delta}\left( \frac{(-k+1/4)_{n}(-k+3/4)_n}{k!^2(n-k)!^2}\right)^{s+2} \\
	&= (-4k+2n)^{\delta} \left( \frac{(4k)!(4n-4k)!}{(2k)!(2n-2k)!k!^2(n-k)!^2} \right)^{s+2}.
\end{align*}
To prove $v_q(a_{n,2s+4,k}) \geqslant s+2$, it suffices to show that
\[ v_q\left( \frac{(4k)!(4n-4k)!}{(2k)!(2n-2k)!k!^2(n-k)!^2} \right) \geqslant 1. \]
Write $x=\{n/q\}$ and $y=\{k/q\}$. We have $1/2 < x <1$. Since $q > \sqrt{10n}$, we have
\begin{align*}
v_q\left( \frac{(4k)!(4n-4k)!}{(2k)!(2n-2k)!k!^2(n-k)!^2} \right) 
&= \lfloor 4y \rfloor + \lfloor 4x-4y \rfloor - \lfloor 2y \rfloor - \lfloor 2x-2y \rfloor - 2\lfloor y \rfloor -2\lfloor x-y \rfloor.  
\end{align*}
It is straightforward to check that the right-hand side above is at least $1$ whenever $1/2<x<1$ and $0\leqslant y < 1$. So the assertion \eqref{asser} is true for $j=0$. 

Suppose now that $1 \leqslant j \leqslant 2s+3$ and the assertion \eqref{asser} is true for any smaller $j$. Let 
\begin{align}
 U(t) :=&~ \frac{\widetilde{A_n}^{\prime}(t)}{\widetilde{A_n}(t)} \notag\\
 =&~ \frac{\delta}{t+n/2} +(s+2)\sum_{\ell=0}^{n-1}\left( \frac{1}{t+\ell+1/4} + \frac{1}{t +\ell +3/4} \right) -(2s+4)\sum_{\ell=0 \atop \ell \neq k}^{n} \frac{1}{t+\ell}. \label{U}
\end{align}
Applying the Leibniz rule, we have 
\begin{align}
	\widetilde{A_n}^{(j)}(-k) &= \left(  \widetilde{A_n}(t) \cdot U(t)\right)^{(j-1)} \big|_{t=-k} \notag\\
	&= \sum_{v=0}^{j-1} \binom{j-1}{v} \widetilde{A_n}^{(j-1-v)}(-k) \cdot U^{(v)}(-k). \label{aaa}
\end{align}
For any $0\leqslant v \leqslant j-1$, by \eqref{eqn_a=deri_new} and the induction hypothesis we have
\begin{equation}\label{bbb}
	 v_q\left( \widetilde{A_n}^{(j-1-v)}(-k) \right) = v_q((j-1-v)!a_{n,2s+4-(j-1-v),k}) \geqslant s+2-(j-1-v).
\end{equation} 
On the other hand, by \eqref{U} we have 
\begin{equation}\label{U^(v)}
U^{(v)}(t) = (-1)^{v}v! \left( \frac{\delta}{(t+n/2)^{1+v}} +\sum_{\ell=0}^{n-1} \frac{s+2}{(t+\ell+1/4)^{1+v}} + \sum_{\ell=0}^{n-1} \frac{s+2}{(t +\ell +3/4)^{1+v}} -\sum_{\ell=0 \atop \ell \neq k}^{n} \frac{2s+4}{(t+\ell)^{1+v}} \right). 
\end{equation}
Since $q>\sqrt{10n}$, we can see from \eqref{U^(v)} that
\begin{equation}\label{ccc}
v_q\left( U^{(v)}(-k)\right) \geqslant -1-v.
\end{equation} 
(We have assumed that $n$ is odd, so $-k+n/2 \neq 0$ and $t=-k$ is not a pole of $U(t)$.)
Therefore, we conclude from \eqref{aaa},\eqref{bbb} and \eqref{ccc} that
\begin{equation}\label{ddd}
v_q\left( \widetilde{A_n}^{(j)}(-k) \right) \geqslant s+2-j. 
\end{equation}
Noting that $q > \sqrt{10n} > 2s+4$ and $j<2s+4$, by \eqref{eqn_a=deri_new} and \eqref{ddd} we obtain
\[ v_q(a_{n,2s+4-j,k}) \geqslant s+2-j, \]
which completes the induction procedure.
\end{proof}

\begin{remark}
	 The conclusion \eqref{strong_ari_a} still holds for any even $n>(2s+4)^2$ by slightly modifying the proof for the case $k=n/2$. Since we will take $n=2^m-1$ eventually, we do not need the even $n$ case. 
\end{remark}

\bigskip

\begin{lemma}\label{Ari_i}
	We have 
	\begin{align}
		\Phi_n^{-s-2}d_n^{2s+4-i} \rho_{n,i} &\in \mathbb{Z}, \quad (1 \leqslant i \leqslant 2s+4) \label{rho_i_ari} \\
		d_n^{s+2-i}\sigma_{n,i} &\in \mathbb{Z}. \quad (1 \leqslant i \leqslant s+2) \label{sigma_i_ari}
	\end{align}
\end{lemma}

\begin{proof}
	The equation \eqref{rho_i_ari} follows from \eqref{def_rho_i} and \eqref{strong_ari_a}. While \eqref{sigma_i_ari} follows from \eqref{def_sigma_i} and \eqref{ari_b}.
\end{proof}

The following lemma is crucial. It makes use of the multiple roots of $A_n(t)$ and $B_n(t)$. Such tricks have been used in \cite{FSZ19,LY20,Spr20}.

\begin{lemma}\label{Ari_0}
	We have 
	\begin{align}
		\Phi_n^{-s-2} d_n^{3s+5}\rho_{n,0} &\in \mathbb{Z}, \label{rho_0_ari}\\
		d_n^{2s+3}\sigma_{n,0} &\in \mathbb{Z}. \label{sigma_0_ari}
	\end{align}
\end{lemma}

\begin{proof}
	Suppose that \eqref{rho_0_ari} is not true. Then by \eqref{def_rho_0} we deduce that, there exist an integer $k_0 \in [1,n]$ and an integer $\ell_0 \in [0,k_0-1]$ such that
	\[ \Phi_n^{-s-2} d_n^{3s+5} \sum_{i=1}^{2s+4} \frac{(-1)^{s+1}(i)_{s+1}a_{n,i,k_0}}{\left(\ell_0+\frac{1}{4}\right)^{i+s+1}} \notin \mathbb{Z}.  \] 
	Noticing that $t=-k_0+\ell_0+1/4$ is a root of $A_n(t)$ with multiplicity $s+2$ by \eqref{def_A}, we deduce that $t=-k_0+\ell_0+1/4$ is a root of
	\[ A_n^{(s+1)}(t) = \sum_{i=1}^{2s+4}\sum_{k=0}^{n} \frac{(-1)^{s+1}(i)_{s+1}a_{n,i,k}}{(t+k)^{i+s+1}}. \] 
	Thus, we obtain
	\[ \Phi_n^{-s-2} d_n^{3s+5} \sum_{i=1}^{2s+4} \frac{(-1)^{s+1}(i)_{s+1}a_{n,i,k_0}}{\left(\ell_0+\frac{1}{4}\right)^{i+s+1}} = -\Phi_n^{-s-2} d_n^{3s+5}\sum_{i=1}^{2s+4}\sum_{0 \leqslant k \leqslant n \atop k \neq k_0} \frac{(-1)^{s+1}(i)_{s+1}a_{n,i,k}}{\left(-k_0+\ell_0+k+\frac{1}{4}\right)^{i+s+1}} \notin \mathbb{Z}.\]
Therefore, there exist $1 \leqslant i_0,i_1\leqslant 2s+4$ and $0 \leqslant k_1 \leqslant n$ with $k_1 \neq k_0$ and a prime $q$ such that	
\[ v_q\left( \Phi_n^{-s-2} d_n^{3s+5} \frac{(-1)^{s+1}(i_0)_{s+1}a_{n,i_0,k_0}}{\left(\ell_0+\frac{1}{4}\right)^{i_0+s+1}} \right) <0,\quad v_q\left( \Phi_n^{-s-2} d_n^{3s+5} \frac{(-1)^{s+1}(i_1)_{s+1}a_{n,i_1,k_1}}{\left(-k_0+\ell_0+k_1+\frac{1}{4}\right)^{i_1+s+1}} \right) <0. \]
On the other hand, we have by \eqref{strong_ari_a} that
\[ v_q\left( \Phi_n^{-s-2}d_n^{2s+4-i_0}a_{n,i_0,k_0} \right) \geqslant 0, \quad v_q\left( \Phi_n^{-s-2}d_n^{2s+4-i_1}a_{n,i_1,k_1} \right) \geqslant 0.  \]
It follows that
\[ v_q\left( \ell_0 + \frac{1}{4} \right) > v_q(d_n), \quad v_q\left(-k_0+ \ell_0 +k_1+ \frac{1}{4} \right) > v_q(d_n),  \]
then
\[ v_q(-k_0+k_1) > v_q(d_n). \]
But this contradicts $0 <|-k_0+k_1|\leqslant n$. So we conclude that \eqref{rho_0_ari} is true.

The proof of \eqref{sigma_0_ari} is similar. The key point is that $B_n^{(s+1)}(t)$ has roots $-1+1/4,-2+1/4,\ldots,-n+1/4$. 
\end{proof}

	

\section{Archimedean norm of the coefficients}

It is well known that the prime number theorem implies 
\begin{equation}\label{d_n_est}
d_n = e^{(1+o(1))n} \quad\text{as~} n \rightarrow \infty. 
\end{equation}
The following lemma is also a corollary of the prime number theorem.

\begin{lemma}\label{Phi_est}
	We have the following asymptotic estimate of the factor $\Phi_n$ defined in \eqref{def_Phi}:
	\[ |\Phi_n| = \exp\left((2\log 2 -1+o(1))n\right) \quad\text{as~} n \rightarrow \infty.\]
\end{lemma}

\begin{proof}
	By the prime number theory, we know that 
	\begin{equation}\label{PNT}
		\sum_{q \leqslant Y} \log q = (1+o(1))Y, \quad \text{as~}Y\rightarrow\infty.
	\end{equation}
We have
\[ \log \Phi_n = \sum_{\sqrt{10n}<q\leqslant n \atop \left\{ \frac{n}{q} \right\}>\frac{1}{2}} \log q = \sum_{q\leqslant n \atop \left\{ \frac{n}{q} \right\}>\frac{1}{2}} \log q + O(\sqrt{n}).\]
It suffices to show that
\[ L_n:= \sum_{q\leqslant n \atop \left\{ \frac{n}{q} \right\}>\frac{1}{2}} \log q = (2\log 2 -1+o(1))n. \]
In fact, we have
\begin{align*}
	L_n &= \sum_{k=1}^{\infty} \sum_{k+\frac{1}{2} < \frac{n}{q} < k+1} \log q \\
	&= \sum_{k=1}^{\infty} \sum_{\frac{n}{k+1} < q < \frac{n}{k+1/2}} \log q.
\end{align*}
It follows that for every integer $K \geqslant 2$ we have
\[ \sum_{k=1}^{K} \sum_{\frac{n}{k+1} < q < \frac{n}{k+1/2}} \log q \leqslant L_n \leqslant \sum_{k=1}^{K} \sum_{\frac{n}{k+1} < q < \frac{n}{k+1/2}} \log q + \sum_{q < \frac{n}{K+3/2}} \log q. \]
Then \eqref{PNT} implies that 
\[ n\sum_{k=1}^{K} \left( \frac{1}{k+1/2} - \frac{1}{k+1} \right) +o(n) \leqslant L_n \leqslant n\sum_{k=1}^{K} \left( \frac{1}{k+1/2} - \frac{1}{k+1} \right) + \frac{n}{K+3/2} +o(n). \]
Letting $K \rightarrow \infty$ we obtain
\[ L_n = n\sum_{k=1}^{\infty} \left( \frac{1}{k+1/2} - \frac{1}{k+1} \right) + o(n) = (2\log 2 -1)n +o(n),\]
as desired.
\end{proof}

\begin{lemma}\label{rhosigma_est}
	As $n \rightarrow \infty$, we have
	\begin{align}
		\max_{0 \leqslant i \leqslant 2s+4} |\rho_{n,i}| &\leqslant 2^{(6s+12+o(1))n}, \label{rho_est} \\
		\max_{0 \leqslant i \leqslant s+2} |\sigma_{n,i}| &\leqslant 2^{(3s+6+o(1))n}. \label{sigma_est}
	\end{align}
\end{lemma}

\begin{proof}
	We only prove \eqref{rho_est}. The proof of \eqref{sigma_est} is similar. By \eqref{def_rho_0} and \eqref{def_rho_i}, we have
	\[ \max_{0 \leqslant i \leqslant 2s+4} |\rho_{n,i}| \leqslant (3s+5)!4^{3s+5} (n+1)^2 \max_{1\leqslant i \leqslant 2s+4 \atop 0 \leqslant k \leqslant n} |a_{n,i,k}|. \]
	It suffices to show that
	\[ \max_{i,k} |a_{n,i,k}| \leqslant 2^{(6s+12+o(1))n}. \]
	
By Cauchy's integral formula, we have
\[ a_{n,i,k} = \frac{1}{2\pi\sqrt{-1}}\int_{|t+k|=\frac{1}{8}} (t+k)^{i-1} A_n(t) \mathrm{d}t, \]
and hence
\begin{align}
|a_{n,i,k}| &\leqslant 2^{(6s+12)n} \sup_{|t+k|=\frac{1}{8}} \left| (t+k)^{i-1}(4t+2n)^{\delta} \frac{\left(t+\frac{1}{4}\right)_{n}^{s+2}\left(t+\frac{3}{4}\right)_{n}^{s+2}}{(t)_{n+1}^{2s+4}}\right| \notag\\
&\leqslant 2^{(6s+12)n} \cdot (10n) \sup_{|t+k|=\frac{1}{8}} \left| \left(t+\frac{1}{4}\right)_{n}^{s+2}\left(t+\frac{3}{4}\right)_{n}^{s+2} (t)_{n+1}^{-2s-4}\right|. \label{eqn_5_4}
\end{align}
In the following, we estimate each term in \eqref{eqn_5_4}. We have obviously
\begin{align*}
	\sup_{|t+k|=\frac{1}{8}} \left| \left(t+\frac{1}{4}\right)_{n} \right| &\leqslant k!(n-k)!, \\
	\sup_{|t+k|=\frac{1}{8}} \left| \left(t+\frac{3}{4}\right)_{n} \right| &\leqslant k!(n-k)!, \\
	\sup_{|t+k|=\frac{1}{8}} \left| \left(t\right)_{n+1}^{-1} \right| &\leqslant \frac{100n^2}{k!(n-k)!}.
\end{align*}
Therefore, 
\[ \max_{i,k}|a_{n,i,k}| \leqslant 2^{(6s+12)n} (10n) (100n^2)^{2s+4},  \]
as desired.
\end{proof}

\section{$2$-adic norm of the linear forms}

\begin{lemma}\label{Lem_5_1}
	Let $m \geqslant 2$ be an integer and $n = 2^m-1$. Let $k_0 = 2^{m-1}$. Then, for any integer $k$ such that $1 \leqslant k \leqslant n$ and $k \neq k_0$, we have
	\[v_2((k-1)!(n-k)!) \geqslant v_2((k_0-1)!(n-k_0)!) + 1.\]
\end{lemma}
\begin{proof}
	It is well-known that $v_2(a!) = a - \operatorname{sod}_2(a)$ for any nonnegative integer $a$, where $\operatorname{sod}_2(a)$ is the sum of digits of $a$ in its binary expansion. We have
	\begin{align*}
		v_2((k-1)!(n-k)!) &= k-1-\operatorname{sod}_2(k-1) + n-k - \operatorname{sod}_2(n-k) \\
		&= n-1-\operatorname{sod}_2(k-1)-\operatorname{sod}_2(n-k).
	\end{align*}
It suffices to prove that
\[\operatorname{sod}_2(k-1)+\operatorname{sod}_2(n-k) \leqslant \operatorname{sod}_2(k_0-1)+\operatorname{sod}_2(n-k_0) - 1 = 2m-3\]
for $k \neq k_0$. 

For any nonnegative integer $a$, we have $a \geqslant 1+2+2^2+\cdots+2^{\operatorname{sod}_2(a)-1}=2^{\operatorname{sod}_2(a)}-1$ because the number of $1$'s in the binary expansion of $a$ is exactly $\operatorname{sod}_2(a)$. In particular:
\begin{itemize}
	\item $\operatorname{sod}_2(a) \leqslant m-1$ for any integer $a \in [0,2^{m}-2]$.
	\item If $\operatorname{sod}_2(a) = m-1$ and $a \in [0,2^{m-1}-1]$, then $a = 2^{m-1}-1$.
\end{itemize}
Now, since both $k-1$ and $n-k$ belong to $[0,2^{m}-2]$, we have $\operatorname{sod}_2(k-1)+\operatorname{sod}_2(n-k) \leqslant (m-1)+(m-1) = 2m-2$. If the equality holds, then $\operatorname{sod}_2(k-1) = \operatorname{sod}_2(n-k) = m-1$. By noticing that $\min\{k-1,n-k\} \in [0,2^{m-1}-1]$, the equality $\operatorname{sod}_2(k-1) = \operatorname{sod}_2(n-k) = m-1$ implies $\min\{k-1,n-k\} = 2^{m-1}-1 \Rightarrow k=k_0$. So we have 
\[\operatorname{sod}_2(k-1)+\operatorname{sod}_2(n-k) < 2m-2\]
for $k \neq k_0$, as desired.  
\end{proof}

\begin{lemma}\label{Lem_nonvanishing_A}
	For any integer $n$ of the form $n=2^m-1$ with $m \geqslant 2$, we have
	\[v_2\left( \Phi_n^{-s-2}d_n^{3s+5} S_n \right) = (10s+20)n + (s+2)m + 2s +v_2((s+2)!)+2.\]
	In particular, $\Phi_n^{-s-2}d_n^{3s+5}S_n \neq 0$ and 
	\[ \left| \Phi_n^{-s-2}d_n^{3s+5}S_n \right|_2 = 2^{(-10s-20+o(1))n} \]
	as $n =2^m-1 \rightarrow \infty$.
\end{lemma}

\begin{proof}
	Obviously $v_2(\Phi_n) = 0$ and $v_2(d_n)=m-1$. It suffices to prove that
	\begin{equation*}
		v_2(S_n) = (10s+20)n - (2s+3)m+5s+v_2((s+2)!)+7.
	\end{equation*}
Note that 
\[
	A_n\left(t+\frac{1}{4}\right) = 2^{(9s+18)n+4s+8} \cdot f(t),
\]
where 
\[ f(t) = (t+1)^{s+2}(t+2)^{s+2}\cdots(t+n)^{s+2}g(t)\]
and 
\[g(t) = (4t+2n+1)^{\delta} \cdot \frac{\prod_{k=0}^{n-1}(2t+2k+1)^{s+2}}{\prod_{k=0}^{n}(4t+4k+1)^{2s+4}}. \]
So we have
\[
S_n = 2^{(9s+18)n+4s+8} \cdot \int_{\mathbb{Z}_2} f^{(s)}(t)\mathrm{d}t
\]
and it suffices to prove that 
\begin{equation}\label{eqn_2-adic-order-esti}
	v_2\left( \int_{\mathbb{Z}_2} f^{(s)}(t)\mathrm{d}t \right) = (s+2)n -(2s+3)m + s +v_2((s+2)!) -1.
\end{equation}

We define the index set $I$ by
\[I = \left\{ (i_1,\ldots,i_n,j) \in \left(\mathbb{Z}_{\geqslant 0}\right)^{n+1}  ~\mid~ i_1+\cdots+i_n+j = s \right\}.\]
Applying the Leibniz rule, we have
\begin{align}
	f^{(s)}(t) =& \sum_{(i_1,\ldots,i_n,j) \in I} \frac{s!}{i_1! \cdots i_n!} \cdot \left((t+1)^{s+2}\right)^{(i_1)} \cdots \left((t+n)^{s+2}\right)^{(i_n)} \cdot \frac{g^{(j)}(t)}{j!} \notag\\
	=& \sum_{(i_1,\ldots,i_n,j) \in I} \frac{s!}{i_1! \cdots i_n!} \cdot \frac{(s+2)!}{(s+2-i_1)!} \cdots \frac{(s+2)!}{(s+2-i_n)!} \cdot (t+1)^{s+2-i_1} \cdots (t+n)^{s+2-i_{n}} \cdot \frac{g^{(j)}(t)}{j!} \notag\\
	=& \sum_{(i_1,\ldots,i_n,j) \in I} s!\binom{s+2}{i_1} \cdots \binom{s+2}{i_n} \cdot (t+1)^{s+2-i_1} \cdots (t+n)^{s+2-i_n} \cdot \frac{g^{(j)}(t)}{j!} \notag\\
	=& \sum_{(i_1,\ldots,i_n,j) \in I} f_{(i_1,\ldots,i_n,j)}(t), \label{eqn_Sum}
\end{align}
where
\begin{align*}
	f_{(i_1,\ldots,i_n,j)}(t) =& ~ s!\binom{s+2}{i_1} \cdots \binom{s+2}{i_n} \cdot \frac{g^{(j)}(t)}{j!} \cdot n!^{2+j} \binom{t+n}{n}^{2+j} \notag\\
	&\qquad\qquad \times \prod_{k=1}^{n} \left((k-1)!(n-k)! \binom{t+k-1}{k-1}\binom{t+n}{n-k}\right)^{i_k}. 
\end{align*}
We will show that the term corresponding to the index $(0,\ldots,0,i_{2^{m-1}}=s,0,\ldots,0)$ in \eqref{eqn_Sum} dominates the $2$-adic norm of the Volkenborn integral of $f^{(s)}(t)$. More precisely, we will prove that
\begin{equation}\label{eqn_dominating_term}
	v_2\left(\int_{\mathbb{Z}_2} f_{(0,\ldots,0,i_{2^{m-1}}=s,0,\ldots,0)}(t) \mathrm{d}t\right) = (s+2)n -(2s+3)m + s +v_2((s+2)!) -1,
\end{equation}
and 
\begin{equation}\label{eqn_other_terms}
	v_2\left(\int_{\mathbb{Z}_2} f_{(i_1,\ldots,i_n,j)}(t) \mathrm{d}t\right) \geqslant (s+2)n -(2s+3)m + s +v_2((s+2)!) 
\end{equation}
for any $(i_1,\ldots,i_n,j) \in I$ with $i_{2^{m-1}} \neq s$.

Once \eqref{eqn_dominating_term} and \eqref{eqn_other_terms} have been established, they together with \eqref{eqn_Sum} will imply \eqref{eqn_2-adic-order-esti}. Then the proof of Lemma \ref{Lem_nonvanishing_A} will be complete.

\medskip

We first prove \eqref{eqn_dominating_term}. We have
\begin{align*}
&f_{(0,\ldots,0,i_{2^{m-1}}=s,0,\ldots,0)}(t) \\
=&~ \frac{(s+2)!}{2} \cdot (2^m-1)!^2 \cdot (2^{m-1}-1)!^{2s} \cdot g(t) \cdot \binom{t+2^m-1}{2^m-1}^2 \binom{t+2^{m-1}-1}{2^{m-1}-1}^{s}\binom{t+2^m-1}{2^{m-1}-1}^{s}. 
\end{align*}
We have trivially $v_2((2^m-1)!) = n-m$ and $v_2((2^{m-1}-1)!) = \frac{n}{2}-m+\frac{1}{2}$. Therefore, \eqref{eqn_dominating_term} is reduced to
\begin{equation}\label{eqn_5_1}
	v_2\left( \int_{\mathbb{Z}_2} g(t)  \binom{t+2^m-1}{2^m-1}^2 \binom{t+2^{m-1}-1}{2^{m-1}-1}^{s}\binom{t+2^m-1}{2^{m-1}-1}^{s} \mathrm{d}t \right) = -m.
\end{equation}

Clearly, we can express $g(t)$ as a power series
\[g(t) = \sum_{k=0}^{\infty} c_k t^k, \]
where $c_k \in 2^{k}\mathbb{Z}_2$ for every $k \geqslant 0$. Then 
\[\frac{g^{(j)}(t)}{j!} = \sum_{k=j}^{
\infty} c_k \binom{k}{j} t^{k-j}, \]
and we have
\begin{equation}\label{eqn_5_2}
\triangle_m\left( \frac{g^{(j)}(t)}{j!} \right) \geqslant \triangle\left( \frac{g^{(j)}(t)}{j!} \right) \geqslant 0  \quad\text{for every~} j \geqslant 0
\end{equation}
by Lemma \ref{Lem_2_4} (1).

By Lemma \ref{Lem_2_4} (3), we have
\begin{align*}
	\triangle_m\left( \binom{t+2^m-1}{2^m-1}^2 \right) &\geqslant -(m-1)+1=-m+2, \\
	\triangle_m\left( \binom{t+2^{m-1}-1}{2^{m-1}-1} \right) &\geqslant -(m-2)=-m+2, \\
	\triangle_m\left( \binom{t+2^m-1}{2^{m-1}-1} \right) &\geqslant -(m-2)=-m+2.
\end{align*}
Then, by Lemma \ref{Lem_2_4} (2) we have 
\[ \triangle_m\left( g(t)  \binom{t+2^m-1}{2^m-1}^2 \binom{t+2^{m-1}-1}{2^{m-1}-1}^{s}\binom{t+2^m-1}{2^{m-1}-1}^{s}  \right) \geqslant -m+2. \]
Therefore, Lemma \ref{Lem_2_3} implies that
\begin{align*}
 &\int_{\mathbb{Z}_2} g(t)  \binom{t+2^m-1}{2^m-1}^2 \binom{t+2^{m-1}-1}{2^{m-1}-1}^{s}\binom{t+2^m-1}{2^{m-1}-1}^{s} \mathrm{d}t \\
 \equiv&~ \frac{1}{2^{m}} \sum_{k=0}^{2^m-1} g(k)\binom{k+2^m-1}{2^m-1}^2 \binom{k+2^{m-1}-1}{2^{m-1}-1}^{s}\binom{k+2^m-1}{2^{m-1}-1}^{s} \pmod{2^{-m+1}\mathbb{Z}_2}.
\end{align*}
Note that $\binom{k+2^m-1}{2^m-1}$ is even for $1 \leqslant k \leqslant 2^m-1$ by either Kummer's or Lucas' theorem and $g(0) \equiv 1 \pmod{2\mathbb{Z}_2}$. We obtain
\[\int_{\mathbb{Z}_2} g(t)  \binom{t+2^m-1}{2^m-1}^2 \binom{t+2^{m-1}-1}{2^{m-1}-1}^{s}\binom{t+2^m-1}{2^{m-1}-1}^{s} \mathrm{d}t \equiv 2^{-m} \pmod{2^{-m+1}\mathbb{Z}_2},\]
which proves \eqref{eqn_5_1}. 

At last, we prove \eqref{eqn_other_terms}. Fix $(i_1,\ldots,i_n,j) \in I$ with $i_{2^{m-1}} \neq s$. Again by Lemma \ref{Lem_2_4} (3), for any integer $k \in [1,n]$ we have
\begin{align*}
	\triangle\left( \binom{t+n}{n} \right) \geqslant -m+1, \\
	\triangle\left( \binom{t+k-1}{k-1} \right) \geqslant -m+1, \\
	\triangle\left( \binom{t+n}{n-k} \right) \geqslant -m+1.
\end{align*}
They together with \eqref{eqn_5_2} and Lemma \ref{Lem_2_4} (2) imply that 
\[ \triangle\left( \frac{g^{(j)}(t)}{j!} \cdot \binom{t+n}{n}^{2+j} \prod_{k=1}^{n} \binom{t+k-1}{k-1}^{i_k}\binom{t+n}{n-k}^{i_k} \right) \geqslant -m+1. \]
By Lemma \ref{Lem_2_3}, we have
\[ v_2\left( \int_{\mathbb{Z}_2} \frac{g^{(j)}(t)}{j!} \cdot \binom{t+n}{n}^{2+j} \prod_{k=1}^{n} \binom{t+k-1}{k-1}^{i_k}\binom{t+n}{n-k}^{i_k} \mathrm{d}t \right) \geqslant -m. \]
Therefore, to prove \eqref{eqn_other_terms}, it suffices to show 
\[ v_2\left( s! \binom{s+2}{i_1} \cdots \binom{s+2}{i_n} n!^{2+j} \prod_{k=1}^{n} \left((k-1)!(n-k)!\right)^{i_k}\right) \geqslant (s+2)n -(2s+2)m + s + v_2((s+2)!). \]
Since the binomial coefficients are integers, it suffices to prove that
\begin{equation}\label{eqn_5_3}
 v_2\left( s! \binom{s+2}{i_{*}} n!^{2+j} \prod_{k=1}^{n} \left((k-1)!(n-k)!\right)^{i_k}\right) \geqslant (s+2)n -(2s+2)m + s + v_2((s+2)!), 
\end{equation}
where $i_{*} = i_{2^{m-1}}$. Recall $i_{*} \neq s$.

By Lemma \ref{Lem_5_1}, we have
\[ v_2\left(  (k-1)!(n-k)! \right) \begin{cases}
	= n-2m+1, &\text{if~} k=2^{m-1}, \\
	\geqslant n-2m+2, &\text{if~} k \neq 2^{m-1}.
\end{cases} \]
Therefore,
\begin{align*}
	&v_2\left(  n!^{2+j} \prod_{k=1}^{n} \left((k-1)!(n-k)!\right)^{i_k}\right) \\
	\geqslant& (2+j)(n-m) + (i_1+\cdots+i_n)(n-2m+2) - i_{*} \\
	=& (2+j)(m-2) + (i_1+\cdots+i_n+j+2)(n-2m+2)-i_{*} \\
	\geqslant& 2(m-2) + (s+2)(n-2m+2) - i_{*}.
\end{align*}
To prove \eqref{eqn_5_3}, it suffices to show that
\[ v_2\left(s!\binom{s+2}{i_{*}}\right) + s -i_{*} -v_2((s+2)!) \geqslant 0. \]
Write $u = s-i_{*}$, then $1 \leqslant u \leqslant s$. It suffices to show that
\[ u + v_2\left( \frac{s!}{(s-u)!(u+2)!} \right) \geqslant 0. \]
Since $\binom{s}{s-u} \in \mathbb{Z}$, it suffices to prove that $u+v_2(u!/(u+2)!) \geqslant 0$, i.e.,
\[ u \geqslant v_2((u+2)(u+1)). \]
This is indeed true for any integer $u \geqslant 1$. In fact, for odd $u \geqslant 1$ we have $v_2((u+2)(u+1)) = v_2(u+1) \leqslant \frac{\log(u+1)}{\log 2} \leqslant u$; for even $u \geqslant 2$ we have $v_2((u+2)(u+1)) = v_2(u+2) \leqslant \frac{\log(u+2)}{\log 2} \leqslant u$.

The proof of Lemma \ref{Lem_nonvanishing_A} is complete.  
\end{proof}

\bigskip

\begin{lemma}\label{Lem_nonvanishing_B}
	For any integer $n$ of the form $n=2^m-1$ with $m \geqslant 2$, we have
	\[v_2\left( d_n^{2s+3} T_n \right) = (6s+12)n + s +v_2((s+2)!). \]
	In particular, $d_n^{2s+3}T_n \neq 0$ and 
	\[ \left| d_n^{2s+3}T_n \right|_2 = 2^{(-6s-12+o(1))n} \]
	as $n =2^m-1 \rightarrow \infty$.
\end{lemma}

\begin{proof}
	The proof is similar to that of Lemma \ref{Lem_nonvanishing_A}. 
	
	Since $v_2(d_n)=m-1$, it suffices to prove that 
	\[v_2(T_n) = (6s+12)n -(2s+3)m +3s + v_2((s+2)!) + 3. \]
	Note that 
	\[B_n\left(t+\frac{1}{4}\right) = 2^{(5s+10)n+2s+4}\cdot f(t), \]
	where 
	\[f(t) = (t+1)^{s+2}(t+2)^{s+2}\cdots(t+n)^{s+2}g(t)\] 
	and
	\[g(t) = \frac{1}{\prod_{k=0}^{n}(4t+4k+1)^{s+2}}.\]
	It suffices to show that
	\[v_2\left( \int_{\mathbb{Z}_2} f^{(s)}(t) \mathrm{d}t \right) = (s+2)n-(2s+3)m+s+v_2((s+1)!)-1.\]
	
	We define the index set $I$ by
	\[I = \left\{ (i_1,\ldots,i_n,j) \in \left(\mathbb{Z}_{\geqslant 0}\right)^{n+1}  ~\mid~ i_1+\cdots+i_n+j = s \right\}.\]
	Applying the Leibniz rule, we have
	\begin{align}
		f^{(s)}(t) 
		=& \sum_{(i_1,\ldots,i_n,j) \in I} s!\binom{s+2}{i_1} \cdots \binom{s+2}{i_n} \cdot (t+1)^{s+2-i_1} \cdots (t+n)^{s+2-i_n} \cdot \frac{g^{(j)}(t)}{j!} \notag\\
		=& \sum_{(i_1,\ldots,i_n,j) \in I} f_{(i_1,\ldots,i_n,j)}(t), \label{eqn_Sum_2}
	\end{align}
	where
	\begin{align*}
		f_{(i_1,\ldots,i_n,j)}(t) =& ~ s!\binom{s+2}{i_1} \cdots \binom{s+2}{i_n} \cdot \frac{g^{(j)}(t)}{j!} \cdot n!^{2+j} \binom{t+n}{n}^{2+j} \notag\\
		&\qquad\qquad \times \prod_{k=1}^{n} \left((k-1)!(n-k)! \binom{t+k-1}{k-1}\binom{t+n}{n-k}\right)^{i_k}. 
	\end{align*}
	We will show that the term corresponding to the index $(0,\ldots,0,i_{2^{m-1}}=s,0,\ldots,0)$ in \eqref{eqn_Sum_2} dominates the $2$-adic norm of the Volkenborn integral of $f^{(s)}(t)$. More precisely, we will prove that
	\begin{equation*}
		v_2\left(\int_{\mathbb{Z}_2} f_{(0,\ldots,0,i_{2^{m-1}}=s,0,\ldots,0)}(t) \mathrm{d}t\right) = (s+2)n -(2s+3)m + s +v_2((s+2)!) -1,
	\end{equation*}
	and 
	\begin{equation*}
		v_2\left(\int_{\mathbb{Z}_2} f_{(i_1,\ldots,i_n,j)}(t) \mathrm{d}t\right) \geqslant (s+2)n -(2s+3)m + s +v_2((s+2)!) 
	\end{equation*}
	for any $(i_1,\ldots,i_n,j) \in I$ with $i_{2^{m-1}} \neq s$.
	
	The rest of the proof is word-by-word the same as the proof of Lemma \ref{Lem_nonvanishing_A}.
\end{proof}

\bigskip

\begin{remark}
	Using similar (and simpler) arguments as in the proofs of Lemma \ref{Lem_nonvanishing_A} and Lemma \ref{Lem_nonvanishing_B}, we can prove that
	\begin{align*}
		\left| \Phi_n^{-s-2}d_n^{3s+5}S_n \right|_2 &\leqslant 2^{(-10s-20+o(1))n} \\
		\left| d_n^{2s+3}T_n \right|_2 &\leqslant 2^{(-6s-12+o(1))n}
	\end{align*}
for \emph{every} large positive integer $n$. However, it seems hard to prove $S_n \neq 0$ and $T_n \neq 0$ for a general positive integer $n$.
\end{remark}

\section{Proofs of the main results}

\begin{proof}[Proof of Theorem \ref{thm2}]
	Combining Lemma \ref{Lin}, Lemma \ref{Ari_i} and Lemma \ref{Ari_0}, we obtain that: $\Phi_n^{-s-2}d_n^{3s+5}S_n$ is a linear combination of $1$, $\zeta_2(j,1/4)$ ($s+3 \leqslant j \leqslant 3s+5$ and $j \equiv s+1+\delta \pmod{2}$) with integer coefficients.
	
	By \eqref{d_n_est}, Lemma \ref{Phi_est}, Lemma \ref{rhosigma_est} and Lemma \ref{Lem_nonvanishing_A}, when $n=2^m-1 \rightarrow \infty$ we have
	\[ \max_{0\leqslant i \leqslant 2s+4} \left| \Phi_n^{-s-2}d_n^{3s+5} \rho_{n,i} \right| \cdot \left| \Phi_n^{-s-2}d_n^{3s+5}S_n \right|_2 \leqslant \exp\left(\left( (4-6\log2)s+7-12\log2 +o(1)\right)n\right) \rightarrow 0, \]
	(because $4-6\log 2<0$ and $7-12\log2 <0$) and importantly
	\[ \Phi_n^{-s-2}d_n^{3s+5}S_n \neq 0. \]
		
	Applying Lemma \ref{Lem_2_1} to the sequence of linear forms $\{ \Phi_n^{-s-2}d_n^{3s+5}S_n \}_{n=2^m-1, m\geqslant 2}$, we deduce that the following set contains at least one irrational number:
	\[\left\{ \zeta_2\left(j,\frac{1}{4}\right) ~\mid~ j \in\mathbb{Z} \cap [s+3,3s+5], ~j \equiv s+1+\delta \pmod{2} \right\}.\]
	
	This is true for any nonnegative integer $s$ and any $\delta \in \{0,1\}$. The proof of Theorem \ref{thm2} is complete.
\end{proof}

\bigskip

\begin{proof}[Proof of Theorem \ref{thm3}]
	Combining Lemma \ref{Lin}, Lemma \ref{Ari_i} and Lemma \ref{Ari_0}, we obtain that: $d_n^{2s+3}T_n$ is a linear combination of $1$, $\zeta_2(j,1/4)$ ($s+3 \leqslant j \leqslant 2s+3$) with integer coefficients.
	
	By \eqref{d_n_est}, Lemma \ref{rhosigma_est} and Lemma \ref{Lem_nonvanishing_B}, when $n=2^m-1 \rightarrow \infty$ we have
	\[ \max_{0\leqslant i \leqslant s+2} \left| d_n^{2s+3} \sigma_{n,i} \right| \cdot \left| d_n^{2s+3}T_n \right|_2 \leqslant \exp\left(\left( (2-3\log2)s+3-6\log2 +o(1)\right)n\right) \rightarrow 0, \]
	and importantly
	\[ d_n^{2s+3}T_n \neq 0. \]
	
	Applying Lemma \ref{Lem_2_1} we deduce that the following set contains at least one irrational number:
	\[\left\{ \zeta_2\left(j,\frac{1}{4}\right) ~\mid~ j \in\mathbb{Z}\cap[s+3,2s+3] \right\}.\]
	
	This is true for any nonnegative integer $s$. The proof of Theorem \ref{thm3} is complete.
\end{proof}

\bigskip

\begin{proof}[Proof of Theorem \ref{thm1}]
	By Lemma \ref{Lem_2_7}, Theorem \ref{thm1} is the special case $\delta = 1$ of Theorem \ref{thm2}. 
	
	Taking $s=3$ we obtain that at least one of $\zeta_2(7),\zeta_2(9),\zeta_2(11),\zeta_2(13)$ is irrational.
\end{proof}

\bigskip

%
%

\bigskip

\vspace*{3mm}
\begin{flushright}
	\begin{minipage}{148mm}\sc\footnotesize
		{\it E--mail address}: {\tt lilaimath@gmail.com} \vspace*{3mm}
	\end{minipage}
\end{flushright}

\end{document}